\documentclass[11pt,a4paper]{article}

\usepackage[english]{babel}
\usepackage{times}
\usepackage{graphicx}
\usepackage{amscd}
\usepackage{amsmath}
\usepackage{amsrefs}
\usepackage{amsfonts}
\usepackage{amssymb}
\usepackage{amsthm}
\usepackage{latexsym}

\newtheorem{theorem}{Theorem}
\newtheorem{proposition}[theorem]{Proposition}
\newtheorem{lemma}[theorem]{Lemma}
\newtheorem{corollary}[theorem]{Corollary}

\newcommand{\R}{\mathbb{R}}
\newcommand{\N}{\mathbb{N}}
\renewcommand{\epsilon}{\varepsilon}
\newcommand{\eps}{\varepsilon}
\newcommand{\e}{\varepsilon}
\newcommand{\dpic}{d}

\renewcommand{\leq}{\leqslant}
\renewcommand{\le}{\leqslant}

\renewcommand{\ge}{\geqslant}

\usepackage{authblk}

\begin{document}

\author[1,2]{Serena Dipierro}

\author[3]{Ovidiu Savin}

\author[1,2,4,5]{Enrico Valdinoci}

\affil[1]{\footnotesize School of Mathematics and Statistics,
University of Melbourne,
Richard Berry Building,
Parkville VIC 3010,
Australia \medskip } 

\affil[2]{\footnotesize Dipartimento di Matematica, Universit\`a degli studi di Milano,
Via Saldini 50, 20133 Milan, Italy \medskip }

\affil[3]{\footnotesize Department of Mathematics, Columbia University,
2990 Broadway,
New York NY 10027, USA \medskip}

\affil[4]{\footnotesize Weierstra{\ss} Institut f\"ur Angewandte
Analysis und Stochastik, Hausvogteiplatz 5/7, 10117 Berlin, Germany \medskip}

\affil[5]{\footnotesize Istituto di Matematica Applicata e Tecnologie Informatiche,
Consiglio Nazionale delle Ricerche,
Via Ferrata 1, 27100 Pavia, Italy \medskip }

\title{Local approximation of arbitrary functions \\ by solutions of nonlocal equations\thanks{It is a pleasure to thank
Xavier Cabr{\'e} for his interesting comments on a preliminary version
of this manuscript.
The
first author has been supported by Alexander von Humboldt Foundation.
The second author has been supported by NSF grant DMS- 1200701. The third author has been supported by ERC grant 277749 ``EPSILON Elliptic Pde’s and Symmetry of Interfaces and Layers for Odd Nonlinearities'' and PRIN grant 201274FYK7 ``Aspetti variazionali e perturbativi nei problemi differenziali nonlineari''.
Emails:
{\tt serydipierro@yahoo.it}, {\tt savin@math.columbia.edu}, {\tt enrico@mat.uniroma3.it} }}

\date{}

\maketitle

\begin{abstract}
We show that any function can be locally approximated
by solutions of prescribed linear equations of nonlocal type.
In particular, we show that every function is locally $s$-caloric, up to a small error.
The case of non-elliptic and non-parabolic operators is taken into account as well.
\end{abstract}

\bigskip\bigskip

{\footnotesize

}\bigskip

\section{Introduction}

In this paper, we will show that an arbitrary function can be locally
approximated, in the smooth sense, by $s$-caloric functions,
i.e. by solutions of the fractional heat equation in which
the diffusion is due to the $s$-power of the Laplacian, with~$s\in(0,1)$.

The precise result obtained is the following:

\begin{theorem}\label{CAS}
Let~$B_1\subset\R^n$ be the unit ball, $s\in(0,1)$, $k\in\N$ and 
$f:B_1\times(-1,1)\to\R$, with
$f\in C^k(\overline{B_1}\times[-1,1])$. 

Fix $\varepsilon>0$. Then there exists
$u_\epsilon=u\in C^\infty(B_1\times(-1,1))\cap C(\R^{n+1})$ which is compactly
supported in~$\R^{n+1}$ and
such that the following properties hold true:
\begin{eqnarray}
&& \label{i7EQ}
\partial_t u+(-\Delta)^s u=0\, {\mbox{ in $B_1\times(-1,1)$}}\\
{\mbox{and }} && 
\| u-f\|_{C^k(B_1\times(-1,1))}\le \varepsilon.\end{eqnarray}
\end{theorem}

We remark that the approximation result in Theorem~\ref{CAS}
reflects a purely nonlocal phenomenon, since in the local case the
solutions of the classical heat equation are particularly ``rigid''.
For example, solutions of the classical heat equation (i.e. solutions
of equation~\eqref{i7EQ}
when~$s=1$)
satisfy a local Harnack inequality which
prevents arbitrary oscillations (in particular,
these solutions cannot approximate a given function which
does not satisfy these oscillation constraints).
\medskip

On the contrary, in the nonlocal setting, solutions
of linear equations are flexible enough to approximate any given
function, and this approximation results hold true in a very general
context. As a matter of fact,
in our setting, Theorem~\ref{CAS} is just a particular case
of a much more general result that we provide in the forthcoming
Theorem~\ref{MAIN}.

To state this general theorem, we introduce now some
specific notation.
We will often use small fonts to denote ``local variables'',
capital fonts to denote ``nonlocal variables'', and Greek fonts to denote
the set of local and nonlocal variables altogether, namely\footnote{If ${\dpic}=0$, simply there are no ``local variables'' $(x_1,\dots,x_{\dpic})$
to take into account and $\nu=n_1+\dots+n_N$.} given
${\dpic}\in\N$, with ${\dpic}\ge0$, and $N\in\N$, with $N\ge1$,
we consider $x:=(x_1,\dots,x_{\dpic})\in\R^{\dpic}$ and $X
:=(X_1,\dots,X_N)\in\R^{n_1}\times\dots\times\R^{n_N}$
and we let $(x,X)\in \R^\nu$, with $\nu:=\dpic+n_1+\dots+n_N$.
To avoid confusions, when necessary, the $k$-dimensional unit ball will be denoted by $B_1^k$
(of course, when no confusion is possible, we will adopt the usual notation $B_1$).

Given $m=(m_1,\dots,m_{\dpic})\in\N^{\dpic}$ and $(a_1,\dots,a_{\dpic})\in
\R^{\dpic}\setminus\{0\}$,
we consider
the local operator
\begin{equation}\label{PAG81} 
\ell := \sum_{j=1}^{\dpic} a_j \partial^{m_j}_{x_j} .\end{equation}
Also, given $s=(s_1,\dots,s_N)\in (0,1)^N$ and $A=(A_1,\dots,A_N)\in\R^N\setminus\{0\}$,
we consider the nonlocal operator
\begin{equation}\label{PAG82} 
{\mathcal{L}} := \sum_{j=1}^N A_j\,(-\Delta_{X_j})^{s_j},\end{equation}
where we denoted by $(-\Delta_{X_j})^{s_j}$ the fractional Laplacian of order $s_j\in(0,1)$
in the set of variables $X_j\in\R^{n_j}$, namely
\begin{eqnarray*}
&&(-\Delta_{X_j})^{s_j} u(x,X_1,\dots,X_j,\dots,X_N)
\\ &&\qquad:= 
C(n_j,s_j)\,
\lim_{\varrho\searrow0}\limits\int_{Y\in\R^{n_j}\setminus B_\varrho^{n_j}}
\frac{u(x,X_1,\dots,X_j,\dots,X_N)-u(x,X_1,\dots,X_j+Y,\dots,X_N)
}{|Y|^{n_j+2s_j}}\,dY ,\end{eqnarray*}
where we used the normalized constant
$$ C(n_j,s_j):=\frac{4^{s_j}\,s_j\,\Gamma\left(\frac{n_j}{2}+s_j\right)}{\pi^{\frac{n_j}{2}}
\Gamma\left(1-s_j\right)},$$
being $\Gamma$ the Euler's $\Gamma$-function.

Then, we deal with the superposition\footnote{Of course, if $d=0$,
i.e. if there are no ``local variables'', the operator~$\Lambda$
in~\eqref{L8923421} coincides with the purely
nonlocal operator~$ {\mathcal{L}}$.}
of the local and the nonlocal operators, given by
\begin{equation}\label{L8923421} \Lambda:= \ell+{\mathcal{L}}\end{equation}
and we establish that {\em all functions are locally $\Lambda$-harmonic up to
a small error}, i.e. the functions in the kernel of the operator $\Lambda$
are locally dense in $C^k$. The precise result goes as follows:

\begin{theorem}\label{MAIN}
Let $k\in\N$ and $f:\R^\nu\to\R$, with
$f\in C^k(\overline{B_1^\nu})$. Fix $\varepsilon>0$. Then there exist
$u\in C^\infty(B_1^\nu)\cap C(\R^{\nu})$ and $R>1$ such that the following properties hold true:
\begin{eqnarray}
&& \label{TH:M1}
\Lambda u=0\, {\mbox{ in }} \,B_1^\nu,\\
&& \label{TH:M2}
\| u-f\|_{C^k(B_1^\nu)}\le \varepsilon\\
{\mbox{and }}
\label{TH:M3}
&& u=0 \,{\mbox{ in }}\, \R^\nu\setminus B_R^\nu.
\end{eqnarray}
\end{theorem}

It is interesting to remark that not only
Theorem~\ref{MAIN}
immediately implies Theorem~\ref{CAS} as a particular case,
but also that Theorem~\ref{MAIN} does not require any ellipticity
or parabolicity on the operator, which is perhaps a rather surprising fact. Indeed, we stress that
Theorem~\ref{MAIN} is valid also for
operators with hyperbolic structures, and comprises the cases when
$$ \Lambda = \sum_{j=1}^{\dpic} \partial^2_{x_j} + (-\Delta_{X_1})^{s_1}$$
and when
$$ \Lambda = (-\Delta_{X_1})^{s_1}-(-\Delta_{X_2})^{s_2}$$
with $s_1$, $s_2\in(0,1)$.
In this sense, the nonlocal
features of the fractional Laplacian in some variables
dominate the possible
elliptic/parabolic/hyperbolic structure of the operator.
\medskip

The first result in the direction of Theorem \ref{MAIN} has been recently obtained in
\cite{all-fcts-are-sh},
where Theorem \ref{MAIN} was proved in the special case in which $d=0$ and $N=1$
(that is, when there are no ``local variables'' and only one
``nonlocal variable''). Results related to that in \cite{all-fcts-are-sh}
have been obtained in \cite{bucur-caputo} for other types of nonlocal operators,
such as the ones driven by the Caputo derivative.

We also observe that these ``abstract'' approximation results
have also ``concrete'' applications, for instance in mathematical biology:
for example, they show that biological species with nonlocal strategies
can better plan their distribution in order to exhaust a given resource in
a strategic region, thus avoiding any unnecessary waste of resource, see e.g.
\cite{massaccesi, logistic}.\medskip

In this sense, we mention the following application of
Theorem \ref{CAS}:

\begin{theorem}\label{hjSI} Let $s\in(0,1)$ and $k\in\N$. Fix $\e>0$.
Let $\sigma\in C^k\big(\overline{B_1}\times[-1,1], \;(0,+\infty)\big)$.
Then, there exist
$u_\e$, $\sigma_\e\in C^\infty\big({B_1}\times(-1,1), \;(0,+\infty)\big)\cap C(\R^{n+1})$
which are compactly supported and such that
\begin{eqnarray}
\label{E p1}
&& \partial_t u_\e+(-\Delta)^s u_\e 
=(\sigma_\e-u_\e)u_\e \quad{\mbox{ in }} {B_1}\times(-1,1),\\
\label{E p2}
&& u_\e =\sigma_\e \quad{\mbox{ in }} {B_1}\times(-1,1)\\
{\mbox{and }} \label{E p3}
&& \|\sigma-\sigma_\e\|_{C^k\big({B_1}\times(-1,1)\big)}\le \e.
\end{eqnarray}
\end{theorem}

The biological interpretation of Theorem \ref{hjSI} is that $u_\e$
represents the distribution of a population, which satisfies a logistic
equation as given in \eqref{E p1}. The function $\sigma$ can be thought as a resource
(which in turn produces a birth rate proportional to it). The meaning
of Theorem \ref{hjSI} is that, possibly replacing the original 
resource with a slightly different one (as prescribed quantitatively
by \eqref{E p3}), the population can consume all the resource (as given by
\eqref{E p2}).
\medskip

As a matter of fact, from \eqref{E p1} and~\eqref{E p2}, one can write that
$$ \partial_t u_\e+(-\Delta)^s u_\e =0 \quad{\mbox{ in }} {B_1}\times(-1,1).$$
Notice that, in our setting, Theorem \ref{hjSI}
is a simple consequence of
Theorem \ref{CAS} 
(by taking there $f:=\sigma$, and defining~$\sigma_\e:=u_\e$). More general
interactions can also be considered, see e.g. Theorem 1.8 in \cite{logistic}.\medskip

The rest of the paper is organized as follows.

Section \ref{SPAN} contains the main argument towards the proof of Theorem \ref{MAIN},
that is that solutions of nonlocal equations can span the largest possible
space with their derivatives (we remark that this is a purely nonlocal argument,
since, for instance, harmonic functions obviously
cannot span strictly positive second derivatives). The argument to prove
this fact is based on a ``separation of variables'' method. Namely,
we will look for solutions of nonlocal equations in the form of products
of functions depending on ``local'' and ``nonlocal variables''.
The nonlocal part of the function is built by the eigenfunctions
of the nonlocal operators (whose boundary behavior is somehow singular
and can be quantified by the estimates of the previous sections), while
the local part of the function is constructed by an ordinary differential
equation which is designed to compensate all the coefficients of the operator
in the appropriate way.

The proof of Theorem \ref{MAIN} is then discussed step by step,
first in Section \ref{MAonaosec}, where $f$ is supposed to be a monomial,
then in Section \ref{OJLA:S}, where $f$ is supposed to be a polynomial,
and finally completed in the general case in Section \ref{puaoisd687456}.

In the course of the proof of the main results, we
also use a precise boundary behavior of solutions of nonlocal equations:
these estimates depend in turn on some technical
boundary asymptotics
of the Green function of the fractional Laplacian. For the facility
of the reader, the technical proofs of these auxiliary
results are contained in Appendices~A and~B.

\section{Spanning the whole of the Euclidean space with $\Lambda$-harmonic functions}\label{SPAN}

In this section, we show that $\Lambda$-harmonic functions
span the whole
of the Euclidean space (this is a purely nonlocal phenomenon, since, for instance, the second
derivatives of harmonic functions have to satisfy a linear equation, and therefore
are forced to lie in a proper subspace).

For this, to start with, 
we give a precise expansion
of the first eigenfunction of the fractional Laplacian:

\begin{lemma}\label{COR:GREEN2}
Let $e\in \partial B_1$.
Let $\phi_\star$ be the first eigenfunction
for $(-\Delta)^s$, normalized to be positive and
such that $\|\phi_\star\|_{L^2(B_1)}=1$, and let $\lambda_\star>0$ be the corresponding eigenvalue.

Let
\begin{equation}\label{HOES2}
\kappa_\star:= 2^s\,
\kappa(n,s)\, \int_{B_1} \phi_\star(z)\,
\frac{(1-|z|^2)^s }{
s\,|z-e|^n } \,dz\in(0,+\infty).\end{equation}
Then
$$ \lim_{\eps\searrow0} \eps^{|\alpha|-s}  
\partial^\alpha\phi_\star(e+\eps X)= 
(-1)^{|\alpha|}\,\kappa_\star\,\lambda_\star \,s\,(s-1)
\dots(s-|\alpha|+1)\,
e_1^{\alpha_1}\dots e_n^{\alpha_n}\,(-e\cdot X)_+^{s-|\alpha|}\,
$$
in the sense of distribution,
for any~$\alpha\in\N^n$.
\end{lemma}

Not to interrupt the proof of the main results, we postpone the
proof of Lemma~\ref{COR:GREEN2} in Appendix~A.
We remark that upper bounds on~$\kappa_\star$
follow from the regularity results in~\cite{MR3168912}
(for our purposes, we also need to obtain
a precise first order expansion with a nonvanishing
order term).\medskip

To present the proof of our main result, we introduce
some notation.
We consider multi-indices
$i=(i_1,\dots,i_{\dpic})\in\N^{\dpic}$ and $I=(I_1,\dots,I_N)\in \N^{n_1}\times\dots\times\N^{n_N}$.
We will write
\begin{equation}\label{UIAiota}
\iota:=(i,I)=(i_1,\dots,i_{\dpic},I_1,\dots,I_N)\in\N^\nu.\end{equation}
As usual, we set $|\iota|:=i_1+\dots+i_{\dpic}+|I_1|+\dots|I_N|$,
where $|I_1|:=I_{1,1}+\dots+I_{1,n_1}$, and so on.
We also write
$$ \partial^\iota w := \partial^{i_1}_{x_1}\dots\partial^{i_{\dpic}}_{x_{\dpic}}
\partial^{I_1}_{X_1}\dots \partial^{I_N}_{X_N} w.$$
We consider the span of the derivatives of $\Lambda$-harmonic functions,
with derivatives up to order $K$. For this, we denote by
$\partial^K w$ the vector field collecting in its entry all the derivatives
of the form~$\partial^\iota w$ with~$|\iota|\le K$ (in some prescribed
order). Notice that $\partial^K w$ is a vector field
on the Euclidean space~$\R^{K'}$ for some~$K'\in\N$ (of course, $K'$ depends on~$K$).

Then we denote by ${\mathcal{H}}$ the family of all functions
$w\in C(\R^\nu)$ that are compactly supported in $\R^\nu$ and
for which there exists a neighborhood ${\mathcal{N}}$ 
of the origin in $\R^\nu$
such that $w\in C^{\infty}({\mathcal{N}})$
and $\Lambda w=0$ in ${\mathcal{N}}$.

Finally, we define the set
\begin{equation}\label{44bis} {\mathcal{V}}_K \,:=\,
\big\{ \partial^K w(0)
{\mbox{ for all $w\in{\mathcal{H}}$}}\big\}.\end{equation}
By construction ${\mathcal{V}}_K\subseteq\R^{K'}$,
and we have:

\begin{lemma}\label{GERM} It holds that ${\mathcal{V}}_K=\R^{K'}$.
\end{lemma}

\begin{proof} First, we consider the case in which $d\ne0$
(hence, we are taking into account the case in which the ambient space
possesses both
``local'' and ``nonlocal variables''; the case $d=0$ will be then discussed
at the end of the proof).

Since $\Lambda$ is a linear operator, we have that
${\mathcal{V}}_K$ is a vector space, hence a linear subspace of $\R^{K'}$.
So, we argue by contradiction: if ${\mathcal{V}}_K$ does not exhaust the whole
of $\R^{K'}$, then it must lie in a proper subspace. Accordingly, there exists
\begin{equation}\label{uojSUIIaaa}
\vartheta\in \partial B_1^{{K'}}\end{equation}
such that
\begin{equation}\label{uojSUIIaaa:2}
{\mathcal{V}}_K \subseteq \big\{
\zeta\in \R^{K'} {\mbox{ s.t. }} \vartheta\cdot\zeta=0
\big\}.\end{equation}
Now, for any $j\in\{1,\dots,N\}$ we denote by $\tilde\phi_{\star,j}\in C(\R^{n_j})$
the first eigenfunction of $(-\Delta)^{s_j}$ in $B^{n_j}_1$ with Dirichlet 
datum outside $B^{n_j}_1$ (and normalized to have unit norm in~$L^2(\R^{n_j})$).
The corresponding eigenvalue will be denoted by $\lambda_{\star,j}>0$.

We also fix a set of free parameters~$t_1,\dots,t_{\dpic}\in\R$.
We recall
that~$a_j$ and~$A_j$ are the coefficients
of the local and nonlocal parts of the operator, respectively,
as introduced in~\eqref{PAG81}
and~\eqref{PAG82}.
Up to reordering the variables and possibly taking the operators
to the other side of the equation, we suppose that~$A_N>0$
and
we set~$\lambda_j:=\lambda_{\star,j}$ for any~$j\in\{1,\dots,N-1\}$
and
$$ \lambda_N := 
\frac{1}{A_N}\,\left(
\sum_{j=1}^{\dpic} |a_j| t_j^{m_j}-\sum_{j=1}^{N-1} A_j\lambda_j
\right).$$
We also consider the set
$$ {\mathcal{P}}:=
\left\{ t=(t_1,\dots,t_{\dpic})\in\R^{\dpic} {\mbox{ s.t. }}
\sum_{j=1}^{\dpic} |a_j| t_j^{m_j}-\sum_{j=1}^{N-1} A_j\lambda_j>0
\right\}.$$
Notice that ${\mathcal{P}}$
is open and nonvoid (since it contains any point $t$ with large
coordinates $t_1,\dots,t_{\dpic}$). We also remark that for any~$t\in{\mathcal{P}}$
we have~$\lambda_N>0$.

Moreover, by construction
\begin{equation}\label{CONS:1}
\sum_{j=1}^{\dpic} |a_j| t_j^{m_j}-\sum_{j=1}^{N} A_j\lambda_j=0.
\end{equation}
We also set 
$$r_j:=\frac{{\lambda_{\star,j}^{1/2s_j}}}{{{\lambda_j^{1/2s_j}}}}.$$
We notice that, since~$\lambda_{\star,j}=\lambda_{j}$ unless~$j=N$,
we have that~$r_j=1$ for all~$j\in\{1,\dots,N-1\}$.

Also, for any~$j\in\{1,\dots,N\}$, the function 
\begin{equation} \label{UJojlsijhkcIUYT}
\phi_j(X_j):= \tilde\phi_{\star,j}\left( \frac{X_j}{r_j}\right)=
\tilde\phi_{\star,j}\left( \frac{{{{\lambda_j^{1/2s_j}}}}\,X_j}{{{
\lambda_{\star,j}^{1/2s_j}}}}\right)
\end{equation}
is an eigenfunction of $(-\Delta)^{s_j}$ in $B^{n_j}_{r_j}$,
with Dirichlet datum outside $B^{n_j}_{r_j}$ and eigenfunction
equal to $\lambda_j$, that is
\begin{equation}\label{CONS:2}
(-\Delta)^{s_j}\phi_j =\lambda_j\phi_j {\mbox{ in }} B^{n_j}_{r_j}.\end{equation}
Now, we define, for any~$j\in\{1,\dots,{\dpic}\}$,
$$ \bar a_j:=
\left\{ \begin{matrix}
a_j/|a_j| & {\mbox{ if }} a_j\ne0,\\
1 & {\mbox{ if }} a_j=0.
\end{matrix}
\right.$$
We
stress that
\begin{equation}\label{a j not 0}
\bar a_j\neq0.
\end{equation}
Now we
consider, for any~$j\in\{1,\dots,{\dpic}\}$,
the solution of the Cauchy problem
\begin{equation}
\label{varpi10} \left\{ \begin{matrix}
\partial^{m_j}_{x_j} \bar v_j = -\bar a_j \bar v_j,\\
\partial^{i}_{x_j} \bar v_j(0)=1 {\mbox{ for every }} i\in\{0,\dots,m_j-1
\}.\end{matrix}
\right.\end{equation}
Notice that the solution $\bar v_j$ is well defined at least in an interval
of the form $[-\rho_j,\rho_j]$ for a suitable $\rho_j>0$, and we define $$\rho:=\min_{j\in
\{1,\dots,{\dpic}\} }\rho_j.$$
We take
$\bar\tau\in C^\infty_0(B_\rho^{\dpic})$, with $\bar\tau=1$
in $B_{\rho/2}^{\dpic}$, and we set $\tau(x)=\tau(x_1,\dots,x_{\dpic}):=
\bar\tau(t_1x_1,\dots,t_{\dpic}x_{\dpic})$.
Moreover, we introduce the function
$$ v_j(x_j):=
\bar v_j (t_j x_j)  .$$
Notice that
\begin{equation}\label{CONS:3}
a_j \partial^{m_j}_{x_j} v_j = -|a_j| t_j^{m_j} v_j.
\end{equation}
Now, 
we take $e_1,\dots,e_N$, with
\begin{equation}\label{COSTRIZ}
e_j\in\partial B_{r_j}^{n_j}=\partial B_{{{\lambda_{\star,j}^{1/2s_j}}}/{{{\lambda_j^{1/2s_j}}}}}^{n_j}.\end{equation}
We introduce an additional set of free parameters~${Y}_1,\dots,{Y}_N$,
with~${Y}_j\in \R^{n_j}$ and~$e_j\cdot{Y}_j<0$.
We also take~$\eps>0$ (to be taken as small as we wish in the sequel,
possibly in dependence of~$e_1,\dots,e_N$ and~${Y}_1,\dots,{Y}_N$),
and we define
\begin{eqnarray*}
w(x,X)&:=& \tau(x)\,v_1(x_1)
\,\dots\,v_{\dpic}(x_{\dpic})
\,\phi_1(X_1+e_1+\eps{Y}_1)\,\dots\,\phi_N(X_N+e_N+\eps{Y}_N)\\
&=& \tau(x)\,v(x)\,\phi(X),\\
{\mbox{where }}\quad
v(x)&:=&v_1(x_1)
\,\dots\,v_{\dpic}(x_{\dpic}) \\
{\mbox{and }}\quad
\phi(X)&:=&\phi_1(X_1+e_1+\eps{Y}_1)
\,\dots\,\phi_N(X_N+e_N+\eps{Y}_N)
.\end{eqnarray*}
Notice that $w$ is compactly supported in $\R^\nu$. Moreover,
in light of~\eqref{CONS:2} and~\eqref{CONS:3},
if $(x,X)$ is sufficiently close to the origin, we have that
\begin{eqnarray*}
&& \ell w(x,X)= \sum_{j=1}^{\dpic} a_j \partial^{m_j}_{x_j}
\big(\tau(x)\, v(x)\,\phi(X)\big)=
-\sum_{j=1}^{\dpic} |a_j| t_j^{m_j}\, \tau(x)\,v(x)\,\phi(X)\\
{\mbox{and }}
&& {\mathcal{L}} w(x,X)= \sum_{j=1}^{N} A_j (-\Delta)^{s_j}_{X_j}
\Big( \tau(x)\,v(x)\,\phi(X)\Big)=
\sum_{j=1}^{N} A_j \lambda_j
\, \tau(x)\,v(x)\,\phi(X).
\end{eqnarray*}
Hence, by~\eqref{CONS:1},
$$ \Lambda w (x,X)=\ell w(x,X)+{\mathcal{L}} w(x,X)=0$$
if~$(x,X)$ is sufficiently close to the origin.
Consequently, $w\in {\mathcal{H}}$.
Thus, in view of~\eqref{44bis}
and~\eqref{uojSUIIaaa:2}, we have that
\begin{equation}\label{9i5162783940dhvcixzh} 
\begin{split}&
0=\vartheta\cdot \partial^K w(0)=\sum_{|\iota|\le K} \vartheta_\iota \partial^\iota w(0)
\\&\qquad\qquad=\sum_{|i_1|+\dots+|i_{\dpic}|+|I_1|+\dots+|I_N|\le K}
\vartheta_{i_1,\dots,i_{\dpic},I_1,\dots,I_N}
\partial^{i_1}_{x_1}\dots\partial^{i_{\dpic}}_{x_{\dpic}}
\partial^{I_1}_{X_1}\dots \partial^{I_N}_{X_N} w(0).
\end{split}\end{equation}
We claim that, for any $i\in\N$,
\begin{equation}\label{varpi11}
\varpi_{ij}:=\partial^{i}_{x_j} \bar v_j(0)\ne0.
\end{equation}
The proof of this can be done by induction.
Indeed, if $i\in\{0,\dots, m_j-1\}$, then \eqref{varpi11} is true, thanks to the initial condition
in \eqref{varpi10}. Suppose now that 
\begin{equation}\label{bas}
{\mbox{the claim in \eqref{varpi11} holds true for all $i\in\{0,\dots, i_o\}$,
for some $i_o\ge m_j-1$.}}\end{equation}
Then, using the equation in \eqref{varpi10} we have that
\begin{equation}\label{bas2}
\partial^{i_o+1}_{x_j} \bar v_j = \partial^{i_o+1-m_j}_{x_j}
\partial^{m_j}_{x_j} \bar v_j = -\bar a_j \partial^{i_o+1-m_j}_{x_j} \bar v_j.\end{equation}
By \eqref{bas}, we know that $ \partial^{i_o+1-m_j}_{x_j} \bar v_j(0)\ne0$.
This, \eqref{a j not 0} and \eqref{bas2} imply that $\partial^{i_o+1}_{x_j} \bar v_j(0)\ne0$.
This proves \eqref{varpi11}.

Now, from \eqref{varpi11} we have that
$$ \partial_{x_j}^{i_j} v(0) =t_j^{i_j} \varpi_{i_j\, j} \ne0.$$
Hence, we write~\eqref{9i5162783940dhvcixzh} as
\begin{equation}\label{98123gds123sdf} 
\begin{split} 0\;&=
\sum_{|i_1|+\dots+|i_{\dpic}|+|I_1|+\dots+|I_N|\le K}
\vartheta_{i_1,\dots,i_{\dpic},I_1,\dots,I_N}\,
\varpi_{i_1\, 1}\dots\varpi_{i_\dpic\, \dpic}
\,
t_1^{i_1}\dots
t_\dpic^{i_\dpic}\,
\partial^{I_1}_{X_1}\dots \partial^{I_N}_{X_N} \phi(0)\\ &=
\sum_{|i|+|I|\le K}
\vartheta_{i,I}\, \varpi_{i}
t^{i} \partial^{I}_{X} \phi(0)
,
\end{split}\end{equation}
where a multi-index notation has been adopted,
and if $i=(i_1,\dots,i_\dpic)$, $\varpi_{i}:= \varpi_{i_1\, 1}\dots\varpi_{i_\dpic\, \dpic}$.
We stress that
\begin{equation}\label{1267}
\varpi_{i}\ne0,\end{equation} thanks to \eqref{varpi11}.

Recalling~\eqref{UJojlsijhkcIUYT}, we write \eqref{98123gds123sdf} as
\begin{equation} \label{8dgsuigfrytfesdfg}
0=
\sum_{|i|+|I|\le K} 
\vartheta_{i,I}\, \varpi_{i}\,t^{i}
\prod_{j=1}^N 
\left( \frac{\lambda_j}{\lambda_{\star,j}}\right)^{\frac{|I_j|}{2s_j}}
\, \partial^{I_j}_{X_j} \tilde\phi_{\star,j}\left(
\frac{\lambda_j^{1/2s_j}}{\lambda_{\star,j}^{1/2s_j}} \,
(e_j+\eps{Y}_j)\right).\end{equation}
We remark that, in light of~\eqref{uojSUIIaaa},
\begin{equation}\label{NOTALL}{\mbox{$\vartheta_{i,I}$
are not all equal to zero.
}}\end{equation}
Now, by Lemma \ref{COR:GREEN2}, applied to~$s:=s_j$, $\alpha:=I_j$,
$e:=
\frac{\lambda_j^{1/2s_j}}{\lambda_{\star,j}^{1/2s_j}}e_j$ and~$X:=
\frac{\lambda_j^{1/2s_j}}{\lambda_{\star,j}^{1/2s_j}}Y_j$,
we see that, for any~$j\in\{1,\ldots,N\}$,
\begin{equation}\begin{split}\label{speriamo}
& \lim_{\eps\searrow0} \eps^{|I_j|-s_j}  
\partial^{I_j}_{X_j}\phi_{\star,j}\left(\frac{\lambda_j^{1/2s_j} e_j}{
\lambda_{\star,j}^{1/2s_j}}+\eps \frac{\lambda_j^{1/2s_j}Y_j
}{\lambda_{\star,j}^{1/2s_j}}\right)\\=\,& 
(-1)^{|I_j|}\,\kappa_\star\,\lambda_\star \,s_j\,(s_j-1)
\dots(s_j-|I_j|+1)\,
\left( \frac{\lambda_j}{\lambda_{\star,j}}\right)^{\frac{|I_j|}{2s_j}}
e_j^{I_{j}}\,\left(-
\frac{\lambda_j^{1/2s_j}e_j}{\lambda_{\star,j}^{1/2s_j}}
\cdot \frac{\lambda_j^{1/2s_j}Y_j}{\lambda_{\star,j}^{1/2s_j}}\right)_+^{s_j-|I_j|}\,
.\end{split}\end{equation}
Now we consider the multi-index~$\bar I$, with the property
that~$|\bar I|$
is the biggest such that~$|i|+|I|\le K$, and, for such~$\bar I$, we define
$$ \Xi :=\sum_{j=1}^N \left(|\bar{I}_{j}|-s_j\right)
=|\bar I|-\sum_{j=1}^N s_j.$$
Notice that~$\bar I$ is well-defined, thanks to~\eqref{uojSUIIaaa}
and~\eqref{1267}.
As a consequence of~\eqref{speriamo}, 
\begin{equation}\begin{split}\label{forse}&
\lim_{\eps\searrow0} \eps^{\Xi} 
\prod_{j=1}^N 
\left( \frac{\lambda_j}{\lambda_{\star,j}}\right)^{\frac{|\bar{I}_j|}{2s_j}}
\, \partial^{\bar{I}_j}_{X_j} \tilde\phi_{\star,j}\left(
\frac{\lambda_j^{1/2s_j}}{\lambda_{\star,j}^{1/2s_j}} \,
(e_j+\eps{Y}_j)\right)
\\=\,&
\lim_{\eps\searrow0} 
\prod_{j=1}^N \eps^{|\bar{I}_{{j}}|-s_j} 
\left( \frac{\lambda_j}{\lambda_{\star,j}}\right)^{\frac{|\bar{I}_j|}{2s_j}}
\, \partial^{\bar{I}_j}_{X_j} \tilde\phi_{\star,j}\left(
\frac{\lambda_j^{1/2s_j}}{\lambda_{\star,j}^{1/2s_j}} \,
(e_j+\eps{Y}_j)\right)
\\
=\,&
\prod_{j=1}^N \left( \frac{\lambda_j}{\lambda_{\star,j}}\right)^{
\frac{|\bar{I}_j|}{s_j}}
(-1)^{|\bar{I}_j|}\,\kappa_\star\,\lambda_\star \,s_j\,(s_j-1)
\dots(s_j-|\bar{I}_j|+1)\,
e_j^{\bar{I}_{j}}\,\left(-
\frac{\lambda_j^{1/2s_j}e_j}{\lambda_{\star,j}^{1/2s_j}}
\cdot \frac{\lambda_j^{1/2s_j}
Y_j}{\lambda_{\star,j}^{1/2s_j}}\right)_+^{s_j-|\bar{I}_j|}\,
.\end{split}\end{equation}
On the other hand, using again~\eqref{speriamo}, we conclude that,
for any multi-index~$I$, with~$|I|<|\bar I|$, 
\begin{eqnarray*}&&
\lim_{\eps\searrow0} \eps^{\Xi} 
\prod_{j=1}^N 
\left( \frac{\lambda_j}{\lambda_{\star,j}}\right)^{\frac{|{I}_j|}{2s_j}}
\, \partial^{{I}_j}_{X_j} \tilde\phi_{\star,j}\left(
\frac{\lambda_j^{1/2s_j}}{\lambda_{\star,j}^{1/2s_j}} \,
(e_j+\eps{Y}_j)\right)\\
&=& \lim_{\eps\searrow0} \eps^{|\bar I|-|I|} 
\prod_{j=1}^N \eps^{|{I}_{{j}}|-s_j} 
\left( \frac{\lambda_j}{\lambda_{\star,j}}\right)^{\frac{|{I}_j|}{2s_j}}
\, \partial^{{I}_j}_{X_j} \tilde\phi_{\star,j}\left(
\frac{\lambda_j^{1/2s_j}}{\lambda_{\star,j}^{1/2s_j}} \,
(e_j+\eps{Y}_j)\right)
\\&=&0.
\end{eqnarray*}
In consequence of this and~\eqref{forse}, 
after multiplying~\eqref{8dgsuigfrytfesdfg}
by~$\eps^\Xi\in(0,+\infty)$ and sending $\eps\searrow0$,
we obtain
\begin{eqnarray*} 0&=&
\sum_{|i|\le K-|\bar I|}
\vartheta_{i,\bar I} \,\varpi_{i}\,t^i\,
\\ &&\cdot\prod_{j=1}^N \left[
\left( \frac{{{{\lambda_j}}}\, }{{{
\lambda_{\star,j}}}}\right)^{\frac{|\bar{I}_j|}{s_j}}\,
(-1)^{|\bar{I}_j|}\,\kappa_\star\,\lambda_\star \,s_j\,(s_j-1)
\dots(s_j-|\bar{I}_j|+1)\, e_j^{\bar{I}_j}\,
\left(- \frac{{{{\lambda_j^{\frac1{2s_j}}}}}\,
e_j}{{{
\lambda_{\star,j}^{\frac1{2s_j}}}}}\cdot
\frac{{{{\lambda_j^{\frac1{2s_j}}}}}\,
{Y}_j}{{{
\lambda_{\star,j}^{\frac1{2s_j}}}}}\right)_+^{s_j-|\bar{I}_j|}\right]\\
&&=
\sum_{|i|\le K-|\bar I|}
\vartheta_{i,\bar I} \,\varpi_{i}\,t^i\,
\prod_{j=1}^N \left[
\left( \frac{{{{\lambda_j}}}\, }{{{
\lambda_{\star,j}}}}\right)\,
(-1)^{|\bar{I}_j|}\,\kappa_\star\,\lambda_\star \,s_j\,(s_j-1)
\dots(s_j-|\bar{I}_j|+1)\, e_j^{\bar{I}_j}\,
\left(- e_j\cdot
{Y}_j\right)_+^{s_j-|\bar{I}_j|}\right]
.
\end{eqnarray*}
That is, collecting and simplifying some terms, we find that
\begin{equation}\label{FOR:SP78} 0=
\sum_{|i|\le K-|\bar I|}
\tilde\vartheta_{i,\bar I}\,\varpi_{i}\,t^i\,e^{\bar I}\,\prod_{j=1}^N
(-e_j\cdot {Y}_j)_+^{-|\bar{I}_j|}
,\end{equation}
with
\begin{equation}\label{uffi}
\tilde\vartheta_{i,\bar{I}}:=
\vartheta_{i,\bar I} \,
\prod_{j=1}^N  \big( s_j\,(s_j-1)
\dots(s_j-|\bar{I}_j|+1) \big).\end{equation}
Notice that formula~\eqref{FOR:SP78} is true
for any~$(t_1,\dots,t_{\dpic})\in{\mathcal{P}}$, 
any~$e_1,\dots,e_N$ satisfying~\eqref{COSTRIZ},
and any~${Y}_1,\dots,{Y}_N$,
with~${Y}_j\in \R^{n_j}$ and~$e_j\cdot{Y}_j<0$.

For this, we take new free parameters~$T_1,\dots,T_N$ with~$T_j\in\R^{n_j}$ and we choose
$$ 
e_j:=
\frac{{\lambda_{\star,j}^{1/2s_j}}}{{{\lambda_j^{1/2s_j}}}}\cdot
\frac{T_j}{|T_j|}\qquad{\mbox{ and }}\qquad
{Y}_j:= -\frac{T_j}{|T_j|^2}.$$
Then, formula~\eqref{FOR:SP78} becomes
$$ \sum_{|i|\le K-|\bar I|}
\tilde\vartheta_{i,\bar I}\,\varpi_i\,t^i\,T^{\bar I}=0.$$
By the Identity Principle of Polynomials, this gives that
each~$\tilde\vartheta_{i,\bar I} \varpi_i$ is equal to zero. 
Hence, by \eqref{1267}, each~$\tilde\vartheta_{i,\bar I}$
is equal to zero. Therefore, recalling~\eqref{uffi}, we conclude that
each~$\vartheta_{i,\bar I}$
is equal to zero.

Plugging this information into~\eqref{8dgsuigfrytfesdfg}, we obtain that
$$
0=
\sum_{{ {|i|+|I|\le K,}\atop{|I|<|\bar I|} }} 
\vartheta_{i,I}\, \varpi_{i}\,t^{i}
\prod_{j=1}^N 
\left( \frac{\lambda_j}{\lambda_{\star,j}}\right)^{\frac{|I_j|}{2s_j}}
\, \partial^{I_j}_{X_j} \tilde\phi_{\star,j}\left(
\frac{\lambda_j^{1/2s_j}}{\lambda_{\star,j}^{1/2s_j}} \,
(e_j+\eps{Y}_j)\right).$$
Now we choose a multi-index~$\tilde I$, with the property
that~$|\tilde I|$
is the biggest such that~$|i|+|I|\le K$ and~$|\tilde I|<|\bar I|$, 
and, for such~$\tilde I$, we define
$$ \tilde\Xi :=\sum_{j=1}^N \left(|\tilde{I}_{j}|-s_j\right)
=|\tilde I|-\sum_{j=1}^N s_j.$$
Notice again that~$\tilde I$ is well-defined, in virtue
of~\eqref{uojSUIIaaa} and~\eqref{1267}.

Thus, we repeat the same argument as above,
with~$\tilde I$ and~$\tilde\Xi$
in place of~$\bar I$ and~$\Xi$, respectively, and we conclude that
each~$\vartheta_{i,\tilde I}$ is equal to zero.

Iterating this procedure, we obtain that each~$\vartheta_{i,I}$
is equal to zero.
This is in contradiction
with~\eqref{NOTALL}
and so the desired result is established (when $d\ne0$).\medskip

Now we consider the case in which $d=0$,
i.e. when only ``nonlocal variables'' are present.
For this, we argue recursively on $N$ (i.e. on the number of
the ``nonlocal variables'').
When~$N=1$, that is when there is only one set
of ``nonlocal variables'', the result is true, thanks to
Theorem 3.1 in \cite{all-fcts-are-sh}.

Now we suppose that the result is true for~$N-1$
and we prove it for~$N$.
We set
$$ {\mathcal{L}}' := \sum_{j=1}^{N-1} A_j\,(-\Delta_{X_j})^{s_j}
\quad{\mbox{ and }}\quad {\mathcal{L}}_N := A_N\,(-\Delta_{X_N})^{s_N}.$$
We denote by ${\mathcal{H}}'$ the family of all functions
$w'\in C(\R^{n_1+\dots+n_{N-1}})$ that are compactly supported
and for which there exists a neighborhood of the origin
on which $w'$ is smooth and~${\mathcal{L}}'w'=0$.

Similarly, we call~${\mathcal{H}}_N$ the family of all functions
$w_N\in C(\R^{n_N})$ that are compactly supported
and for which there exists a neighborhood of the origin 
on which $w_N$ is smooth and~${\mathcal{L}}_N w_N=0$.

We also use the notation~$X=(X',X_N)\in \R^{n_1+\dots+n_{N-1}}
\times \R^{n_N}$ to distinguish the last set of variables.
Given any~$w'\in {\mathcal{H}}'$
and any~$w_N\in{\mathcal{H}}_N$, we set
$$ W_{w',w_N} (X)=W_{w',w_N} (X',X_N):=
w'(X')\,w_N(X_N).$$
Notice that
$$ {\mathcal{L}} W_{w',w_N} (X) = 
\Big( {\mathcal{L}}'w'(X')\Big)\,w_N(X_N) +
w'(X')\,\Big({\mathcal{L}}_N w_N(X_N)\Big).$$
Thus,
\begin{equation}\label{7yu8rwe9yfuefgffbfbfbfas}
{\mbox{if $w'\in {\mathcal{H}}'$ and $w_N\in{\mathcal{H}}_N$, then
$W_{w',w_N} \in {\mathcal{H}}$.}}
\end{equation}
Again, we argue by contradiction and we suppose that
the claim in Lemma~\ref{GERM} is not true, hence
there exists a unit vector~$\vartheta$
such that 
\begin{equation}\label{7yu8rwe9yfuefgffbfbfbfas2}
{\mbox{${\mathcal{V}}_K$ lies in the orthogonal space of~$\vartheta$.
}}\end{equation}
Notice that each component of~$\vartheta$ can be written as~$\vartheta_I$,
with~$I=(I_1,\dots,I_N)$ and~$|I|\le K$. To distinguish
the last component we write~$I':=(I_1,\dots,I_{N-1})$
and so~$\vartheta_I=\vartheta_{(I',I_N)}$ with~$|I'|+|I_N|\le K$.

In particular, by~\eqref{7yu8rwe9yfuefgffbfbfbfas}
and~\eqref{7yu8rwe9yfuefgffbfbfbfas2}, for any
$w'\in {\mathcal{H}}'$ and any~$w_N\in{\mathcal{H}}_N$
we have that
\begin{eqnarray*}
&& 0=\sum_{|I|\le K} \vartheta_I \partial^I_X W_{w',w_N}(0)=
\sum_{|I'|+|I_N|\le K} \vartheta_{(I',I_N)} \partial^{I'}_{X'} w'(0)
\partial^{I_N}_{X_N} w_N(0)
\\ &&\qquad\qquad= \sum_{|I_N|\le K} \left[
\sum_{|I'|\le K-|I_N|} \vartheta_{(I',I_N)} \partial^{I'}_{X'} w'(0)
\right]\, \partial^{I_N}_{X_N} w_N(0)
= \sum_{|I_N|\le K} \hat\vartheta_{I_N,w'}
\, \partial^{I_N}_{X_N} w_N(0),
\end{eqnarray*}
where
$$ \hat\vartheta_{I_N,w'}:=
\sum_{|I'|\le K-|I_N|} \vartheta_{(I',I_N)} \partial^{I'}_{X'} w'(0).$$
That is, all functions in~$
{\mathcal{H}}_N$ lie in the orthogonal of the vector with
entries~$\hat\vartheta_{I_N,w'}$.
{F}rom Theorem 3.1 in \cite{all-fcts-are-sh}
this implies that each~$\hat\vartheta_{I_N,w'}$ must vanish, that is
$$ \sum_{|I'|\le K-|I_N|} \vartheta_{(I',I_N)} \partial^{I'}_{X'} w'(0)=0$$
for any multi-index~$I_N$ with~$|I_N|\le K$
and any~$w'\in {\mathcal{H}}'$.
Since~${\mathcal{H}}'$ contains $N-1$ ``nonlocal variables'', we can now
use the inductive hypothesis and conclude that each~$\vartheta_{(I',I_N)}$
must vanish. This is a contradiction with the fact that~$\vartheta$
was supposed to be of unit length and so the proof of
Lemma~\ref{GERM} is complete.
\end{proof}

\section{Proof of Theorem \ref{MAIN} when $f$ is a monomial}\label{MAonaosec}

Now we prove Theorem \ref{MAIN} under the additional assumption that $f$ is of monomial type, namely
that
\begin{equation}\label{MAonao}
f(x,X)=
\frac{x_1^{i_1}\dots x_{\dpic}^{i_{\dpic}} X_1^{I_1}\dots X_N^{I_N}}{\iota!}
=\frac{x^i\,X^I}{\iota!},\end{equation}
for some $(i_1,\dots,i_{\dpic})\in\N^{\dpic}$ and $(I_1,\dots,I_N)\in \N^{n_1}\times\N^{n_N}$.
Of course, we used here the standard notation for powers of multi-indices: namely
if $X_1:=(X_{1,1},\dots,X_{1,n_1})\in\R^{n_1}$
and $I_1:=(I_{1,1},\dots,I_{1,n_1})\in\N^{n_1}$, the notation $X_1^{I_1}$
is short for $X_{1,1}^{ I_{1,1} }\dots X_{1,n_1}^{ I_{1,n_1} }$.
Also, $\iota$ is as in \eqref{UIAiota} and, as customary, we used the multi-index factorial
$$ \iota! := i_1!\,\dots\,i_{\dpic}!\,I_1!\,\dots\,I_N!,$$
where, once again $I_1!:=I_{1,1}!\,\dots I_{1,n_1}!$ and so on.

Then, to prove Theorem \ref{MAIN} in this case, we argue as follows. 
We define
\begin{eqnarray}&&\label{dega} 
\gamma:=\sum_{j=1}^{\dpic} \frac{i_j}{m_j}+\sum_{j=1}^N
\frac{|I_j|}{2s_j}\\
{\mbox{and }}&&\label{demu}
\mu:=\min\left\{ \frac{1}{m_1},\,\dots,\,\frac{1}{m_{\dpic}},\,
\frac{1}{2s_1},\,\dots,\,
\frac{1}{2s_N}
\right\}.
\end{eqnarray}
We also take~$K_o\in\N$ with
\begin{equation}\label{89uojk3456ygcxcvaag}
K_o\ge\frac{\gamma+1}{\mu}\end{equation}
and we
let
\begin{equation}\label{deka}
K:=K_o+|i|+|I|+k,\end{equation} where~$k$ is the fixed
integer given by the statement of Theorem \ref{MAIN}.
By Lemma \ref{GERM},
there exist
a neighborhood ${\mathcal{N}}$ 
of the origin in $\R^\nu$ and a function
$w\in C(\R^\nu)$, compactly supported in $\R^\nu$,
such that $w\in C^{\infty}({\mathcal{N}})$,
$\Lambda w=0$ in ${\mathcal{N}}$, and
such that all the derivatives of~$w$ in~$0$ of order up to~$K$
vanish, with the exception of~$\partial^\iota w(0)$, which is equal to~$1$.
In this way, setting
\begin{equation}\label{gdef:01}
g:=w-f, 
\end{equation}
we have that
\begin{equation*}
{\mbox{$\partial^\alpha g(0)=0$ for any~$\alpha\in\N^\nu$ with~$|\alpha|\le K$}}.
\end{equation*}
Accordingly, in ${\mathcal{N}}$ we can write
\begin{equation}\label{gdef:02}
g(x,X)=\sum_{|\tau|\ge K+1} x^t\, X^T\,h_\tau(x,X),\end{equation}
for functions~$h_\tau$ that are smooth in~${\mathcal{N}}$,
where the multi-index notation~$\tau=(t,T)$ has been used.

Now, we fix~$\eta\in(0,1)$ (to be taken suitably small with respect
to the fixed~$\eps>0$ given by the statement of
Theorem \ref{MAIN}). We define
$$u(x,X):=\frac{1}{ \eta^{\gamma}}\, w\big( \eta^{\frac1{m_1}}x_1,\dots,
\eta^{\frac1{m_{\dpic}}}x_{\dpic}, 
\eta^{\frac1{2s_1}}X_1,\dots,\eta^{\frac1{2s_N}}X_N
\big).$$
Notice that~$u$ is compactly supported in~$\R^\nu$
and smooth in a neighborhood of the origin (which is large for~$\eta$
small, hence we may suppose that it includes~$B_1$), and in
this neighborhood we have
\begin{eqnarray*} &&
\eta^\gamma\,\Lambda u(x,X)=\eta\left[
\sum_{j=1}^{\dpic} a_j \partial^{m_j}_{x_j} w\big( \eta^{\frac1{m_1}}x_1,\dots,
\eta^{\frac1{m_{\dpic}}}x_{\dpic}, 
\eta^{\frac1{2s_1}}X_1,\dots,\eta^{\frac1{2s_N}}X_N
\big)\right.\\ &&\qquad\qquad\qquad\left.+
\sum_{j=1}^N A_j\,(-\Delta_{X_j})^{s_j}w\big( \eta^{\frac1{m_1}}x_1,\dots,
\eta^{\frac1{m_{\dpic}}}x_{\dpic}, 
\eta^{\frac1{2s_1}}X_1,\dots,\eta^{\frac1{2s_N}}X_N
\big)\right]
=0.
\end{eqnarray*}
These observations establish~\eqref{TH:M1} and~\eqref{TH:M3}.
Now we prove~\eqref{TH:M2}. To this aim, we observe that
the monomial structure of~$f$ in~\eqref{MAonao} 
and the definition of~$\gamma$ in~\eqref{dega}
imply that
$$
\frac{1}{ \eta^{\gamma}}\, f\big( \eta^{\frac1{m_1}}x_1,\dots,
\eta^{\frac1{m_{\dpic}}}x_{\dpic}, 
\eta^{\frac1{2s_1}}X_1,\dots,\eta^{\frac1{2s_N}}X_N
\big)=f(x,X).$$
Consequently, by~\eqref{gdef:01} and~\eqref{gdef:02},
\begin{eqnarray*}
u(x,X)-f(x,X)&=&
\frac{1}{ \eta^{\gamma}}\, g\big( \eta^{\frac1{m_1}}x_1,\dots,
\eta^{\frac1{m_{\dpic}}}x_{\dpic}, 
\eta^{\frac1{2s_1}}X_1,\dots,\eta^{\frac1{2s_N}}X_N
\big)\\
&=& 
\sum_{|\tau|\ge K+1} 
\eta^{ \left|\frac{t}{m}\right| +\left|\frac{T}{2s}\right|-\gamma }\,
x^t\, X^T\,h_\tau
\big( \eta^{\frac1{m}}x, \eta^{\frac1{2s}}X\big)
\end{eqnarray*}
where the multi-index notation has been used.

Therefore, for any multi-index~$\beta=(b,B)$ with~$|\beta|\le k$,
\begin{equation}\begin{split}\label{8ihjr6yfhytdgdsg}&
\partial^\beta
\left( u(x,X)-f(x,X)\right) =\partial^b_x\,\partial^B_X\,
\left( u(x,X)-f(x,X)\right)\\
&\qquad=\sum_{ {{|b'|+|b''|=|b|}\atop{|B'|+|B''|=|B|}}\atop{|\tau|\ge K+1}}
c_{\tau,\beta}\,
\eta^{ \left|\frac{t}{m}\right| +\left|\frac{T}{2s}\right|-\gamma+
\left|\frac{b''}{m}\right|+\left|\frac{B''}{2s}\right| }\,\,
x^{t-b'}\, X^{T-B'}\,\partial^{b''}_x\,\partial^{B''}_X\,h_\tau
\big( \eta^{\frac1{m}}x, \eta^{\frac1{2s}}X\big),
\end{split}\end{equation}
for suitable coefficients~$c_{\tau,\beta}$.
Thus, to prove~\eqref{TH:M2}, we need to show that
this quantity is small if so is~$\eta$. To this aim,
we use~\eqref{demu}, \eqref{89uojk3456ygcxcvaag} and~\eqref{deka} to see that
\begin{eqnarray*}
&&\left|\frac{t}{m}\right| +\left|\frac{T}{2s}\right|-\gamma+
\left|\frac{b''}{m}\right|+\left|\frac{B''}{2s}\right|\ge 
\left|\frac{t}{m}\right| +\left|\frac{T}{2s}\right|-\gamma\\
&&\qquad\qquad\ge \mu \big( |t|+|T|\big)-\gamma\ge
K\mu-\gamma\ge K_o\mu-\gamma\ge1.\end{eqnarray*}
Consequently, we deduce from~\eqref{8ihjr6yfhytdgdsg} that~$
\| u-f\|_{C^k(B_1^\nu)}\le C\eta$, for some~$C>0$.
By choosing~$\eta$ sufficiently small with respect to~$\eps$,
this implies~\eqref{TH:M2}.
The proof of Theorem \ref{MAIN} when $f$ is a monomial
is thus complete.

\section{Proof of Theorem \ref{MAIN} when $f$ is a polynomial}\label{OJLA:S}

If $f$ is a polynomial, we can write $f$ as a finite sum of monomials, say
$$ f(x,X)=\sum_{j=1}^J c_j f_j(x,X),$$
where each $f_j$ is a monomial as in \eqref{MAonao}, $c_j\in\R$ and $J\in\N$.
Let $c:=\max_{j\in J}|c_j|$.
Then,
we know that Theorem \ref{MAIN}
holds for each $f_j $, in view of the proof given in Section \ref{MAonaosec},
and so
we find $u_j\in C^\infty(B_1^\nu)\cap C(\R^\nu)$ and $R_j>1$ such that
$\Lambda u_j=0 $ in $ B_1^\nu$,
$ \| u_j-f_j\|_{C^k(B_1^\nu)}\le \varepsilon$
and $ u=0 $ in $ \R^\nu\setminus B_{R_j}^\nu$. Hence, we set
$$ u(x,X)=\sum_{j=1}^J c_j u_j(x,X),$$
and we see that
$$ \| u-f\|_{C^k(B_1^\nu)}\le
\sum_{j=1}^J |c_j|\,\| u_j-f_j\|_{C^k(B_1^\nu)}\le cJ\,\varepsilon.$$
Also, since $\Lambda$ is linear, we have that $\Lambda u=0 $ in $ B_1^\nu$.
Finally, $u$ is supported in $B_{R}^\nu$, being $R:=\max_{j\in J}R_j$.
This establishes Theorem \ref{MAIN} for polynomials
(up to replacing $\varepsilon$
with $cJ\,\varepsilon$).

\section{Completion of the proof of Theorem \ref{MAIN}}\label{puaoisd687456}

Let $f$ be as in the statement of Theorem \ref{MAIN}.
By a version of the Stone-Weierstra{\ss} Theorem
(see e.g. Lemma 2.1 in \cite{all-fcts-are-sh}), we know that there exists
a polynomial $\tilde f$ such that $
\| f-\tilde f\|_{C^k(B_1^\nu)}\le \varepsilon$. Then, we know that Theorem \ref{MAIN}
holds for $\tilde f $, in view of the proof given in Section \ref{OJLA:S}, and so
we find $u\in C^\infty(B_1^\nu)\cap C(\R^\nu)$ and $R>1$ such that 
$\Lambda u=0 $ in $ B_1^\nu$,
$ \| u-\tilde f\|_{C^k(B_1^\nu)}\le \varepsilon$
and $ u=0 $ in $ \R^\nu\setminus B_R^\nu$.
Then, we see that $\| u-f\|_{C^k(B_1^\nu)}\le 
\| u-\tilde f\|_{C^k(B_1^\nu)}+\| f-\tilde f\|_{C^k(B_1^\nu)}\le 2\varepsilon$,
hence Theorem \ref{MAIN} is proved (up to replacing $\varepsilon$ 
with $2\varepsilon$).

\begin{appendix}

\section*{Appendix A -- 
Boundary behavior of solutions of fractional Laplace equations
and proof of Lemma \ref{COR:GREEN2}}

In this section, we detect the exact boundary
behavior of solutions of fractional Laplace equations in a 
ball with Dirichlet data, in order to prove Lemma~\ref{COR:GREEN2}.
For estimates in general domains, see e.g. \cite{MR3168912}
and the references therein. Let us remark that, in our context,
we do not only obtain
bounds from above and below, but also a precise asymptotics in the limits
which approach the boundary.

In order to obtain our bounds,
we make use of the fractional Green function, whose setting goes as follows.
Given $s\in(0,1)$ and $x$, $z\in B_1$,
we consider the function
\begin{equation}\label{G def} G(x,z):=|z-x|^{2s-n}\int_0^{r_0(x,z)} \frac{t^{s-1}\,dt}{
(t+1)^{\frac{n}{2}}},\end{equation}
with\footnote{Though we will not use this,
it is interesting to point out that
when $n=2s$, i.e. when $n=1$ and $s=\frac12$, the function $G$
can be written explicitly, up to constants, as
$$ G(x,z)=\log \frac{1-xz+\sqrt{(1-|x|^2)\,(1-|z|^2)}}{|z-x|}.$$
This follows by computing the integral
$$ \int \frac{dt}{ \sqrt{ t\,(t+1)}}=
2\log (\sqrt{t}+\sqrt{t+1})+{\mbox{const}}\,.$$}
\begin{equation}\label{r0 def}
r_0(x,z):=\frac{(1-|x|^2)\,(1-|z|^2)}{|z-x|^2}.\end{equation}
Up to normalization factors,
the function $G$ plays the role of a Green function in the fractional setting,
as discussed for instance in~\cite{MR3461641} and in the references therein.

If $x$ lies in an $\eps$-neighborhood of~$\partial B_1$, then $G$ is
of order $\eps^s$, as stated precisely in the next result:

\begin{lemma}\label{GOALEM}
Let $e\in \partial B_1$, $\eps_o>0$ and $\omega\in \partial B_1$. Assume that $
e+\eps\omega\in B_1$ for all $\eps\in(0,\eps_o]$. Let $f\in C^\alpha(\R^{\dpic})$ for some $\alpha\in(0,1)$,
with $f=0$ outside $B_1$.

Then
\begin{equation} \label{GOA}
\lim_{\eps\searrow0}\eps^{-s}\int_{B_1} f(z)\,G(e+\eps\omega,z)\,dz
=
\int_{B_1} f(z)\,
\frac{(-2\,e\cdot\omega)^s\,(1-|z|^2)^s }{
s\,|z-e|^n } \,dz.\end{equation}
\end{lemma}

The rather technical proof of Lemma~\ref{GOALEM} is postponed
to Appendix~B, for the facility of the reader. Here,
we deduce from Lemma~\ref{GOALEM} the boundary estimates needed
to the proof of our main result:

\begin{proposition}\label{123rtUOJLS}
Let $e\in \partial B_1$, $\eps_o>0$ and $\omega\in \partial B_1$. Assume that $
e+\eps\omega\in B_1$ for all $\eps\in(0,\eps_o]$. Let $f\in C^\alpha(\R^{n})$ for some $\alpha\in(0,1)$,
with $f=0$ outside $B_1$.

Let $u$ be a weak solution of
$$ \left\{ \begin{matrix}
(-\Delta)^s u =f & {\mbox{ in }} B_1,\\
u=0 & {\mbox{ in }}\R^n\setminus B_1.
\end{matrix}\right. $$
Then
\begin{equation*} 
\lim_{\eps\searrow0}\eps^{-s}\,u(e+\eps\omega)
=
\kappa(n,s)\,(-2\,e\cdot\omega)^s\,\int_{B_1} f(z)\,
\frac{(1-|z|^2)^s }{
s\,|z-e|^n } \,dz,\end{equation*}
where
$$ \kappa(n,s):= \frac{\Gamma\left(\frac{n}2\right)}{4^s\,\pi^{\frac{n}2}\,\Gamma^2(s)},$$
being $\Gamma$ the Euler's $\Gamma$-function.
\end{proposition}

\begin{proof}
We know from Theorems 1 and 2 in
\cite{MR3161511} that $u$ is actually continuous in $\R^n$
and it is a viscosity solution of the equation. Also, by the fractional Green Representation Theorem
(see e.g. Theorem 3.2 in \cite{MR3461641} and the references therein), we have that
$$ u(e+\eps\omega)= \kappa(n,s)\,
\int_{B_1} f(z)\,G(e+\eps\omega,z)\,dz,$$
with $G$ as in \eqref{G def}. Hence, the desired result follows from \eqref{GOA}.
\end{proof}

As a simple consequence, we can characterize the boundary behavior of the first eigenfunction
for the fractional Laplacian with Dirichlet data (see e.g. Appendix A in \cite{MR3002745}
for a discussion on fractional eigenvalues).

\begin{corollary}\label{COR:GREEN}
Let $e\in \partial B_1$.
Let $\phi_\star$ be the first eigenfunction
for $(-\Delta)^s$, normalized to be positive and
such that $\|\phi_\star\|_{L^2(B_1)}=1$, and let $\lambda_\star>0$ be the corresponding eigenvalue.
Then,
\begin{equation}\label{HOES}
\|\phi_\star\|_{C^s(\R^n)}\le C,
\end{equation}
for some~$C>0$ depending only on~$n$ and~$s$, and
\begin{equation} \label{ih:AI67811:1}
\lim_{\eps\searrow0}\eps^{-s}\,\phi_\star(e+\eps\omega)
= \kappa_\star \,\lambda_\star\,(-\,e\cdot\omega)^s_+.\end{equation}
where~$\kappa_\star$ is as in~\eqref{HOES2}.
\end{corollary}

\begin{proof} The idea is that,
since $(-\Delta)^s \phi_\star=\lambda_\star \,\phi_\star$,
we can use Proposition \ref{123rtUOJLS} and get the desired result.
More precisely,
we have that $\phi_\star$
is $C^s(B_1)$ (see the proof of Corollary 8 in \cite{MR3161511}
to obtain the continuity and then Proposition 1.1 in \cite{MR3168912}
to get the H\"older estimate in~\eqref{HOES}).

Notice that, by~\eqref{HOES}, we have that the quantity~$\kappa_\star$ 
defined in~\eqref{HOES2} is finite, while
the positivity of $\phi_\star$ implies that $\kappa_\star>0$.

Also,  the H\"older estimate in~\eqref{HOES}
allows to use Proposition \ref{123rtUOJLS}
with $f:=\lambda_\star \,\phi_\star$.
Accordingly,
for any~$\omega\in \partial B_1$ for which there exists~$\eps_o>0$
such that~$
e+\eps\omega\in B_1$ for all $\eps\in(0,\eps_o]$, we have that
\begin{equation} \label{ih:AI67811:2}
\lim_{\eps\searrow0}\eps^{-s}\,\phi_\star(e+\eps\omega)
= \kappa_\star \,\lambda_\star\,(-\,e\cdot\omega)^s.\end{equation}
Now, we distinguish two cases: if~$e\cdot\omega<0$,
then
$$ |e+\eps\omega|^2 =1+2\eps e\cdot\omega+\eps^2<1$$
for small~$\eps$, and so~$e+\eps\omega\in B_1$ for small~$\eps$,
hence~\eqref{ih:AI67811:1} follows from~\eqref{ih:AI67811:2}.

If instead~$e\cdot\omega\ge0$, then~$e+\eps\omega\in\R^n\setminus B_1$,
thus~$\phi_\star(e+\eps\omega)=0$, which obviously implies~\eqref{ih:AI67811:1}
in this case.
\end{proof}

{F}rom Corollary~\ref{COR:GREEN}, we can now complete the proof
of Lemma~\ref{COR:GREEN2}, by arguing as follows:

\begin{proof}[Proof of Lemma~\ref{COR:GREEN2}]
Let~$\psi\in C^\infty_0(\R^n)$.
We write~$X=\rho\omega$, with~$\rho\ge0$ and~$\omega\in
S^{n-1}$. Notice that~$(\eps\rho)^{-s}\,|\phi_\star(e+\eps\rho\omega)|\le C$,
thanks to~\eqref{HOES}.

So, we use~\eqref{ih:AI67811:2} and the Dominated Convergence Theorem
to see that
\begin{eqnarray*}
&& \lim_{\eps\searrow0} \eps^{|\alpha|-s} \int_{\R^n} 
\partial^\alpha\phi_\star(e+\eps X)\, \psi(X)\,dX=
\lim_{\eps\searrow0} \int_{\R^n} 
\partial^\alpha_X \big( \eps^{-s}\phi_\star(e+\eps X)\big)\, \psi(X)\,dX\\
&&\qquad= (-1)^{|\alpha|}
\lim_{\eps\searrow0} \int_{\R^n} 
\eps^{-s}\phi_\star(e+\eps X)\, \partial^\alpha \psi(X)\,dX
\\&&\qquad=
(-1)^{|\alpha|}
\lim_{\eps\searrow0} \int_{0}^{+\infty}\,d\rho\,\int_{S^{n-1}}\,d\omega\, \rho^{n-1}
\rho^s \,(\eps\rho)^{-s}\phi_\star(e+\eps \rho\omega)\, \partial^\alpha \psi(\rho\omega)\\
&&\qquad= (-1)^{|\alpha|} \,\kappa_\star\,\lambda_\star \,
\int_{0}^{+\infty}\,d\rho\,\int_{S^{n-1}}\,d\omega\, \rho^{n-1}
\rho^s \,(-e\cdot\omega)_+^s\, \partial^\alpha \psi(\rho\omega)\\&&\qquad=
(-1)^{|\alpha|} \,\kappa_\star\,\lambda_\star \,
\int_{0}^{+\infty}\,d\rho\,\int_{S^{n-1}}\,d\omega\, \rho^{n-1}
(-e\cdot\rho\omega)^s_+\, \partial^\alpha \psi(\rho\omega)\\
&&\qquad= (-1)^{|\alpha|} \,\kappa_\star\,\lambda_\star \,
\int_{\R^n} 
(-e\cdot X)^s_+\, \partial^\alpha \psi(X)\,dX\\&&\qquad=
\kappa_\star\,\lambda_\star \,
\int_{\R^n} 
\partial^\alpha_X (-e\cdot X)^s_+\, \psi(X)\,dX
\\&&\qquad= 
(-1)^{|\alpha|}\,\kappa_\star\,\lambda_\star \,s\,(s-1)\dots(s-|\alpha|+1)\,
e_1^{\alpha_1}\dots e_n^{\alpha_n}\,\int_{\R^n} (-e\cdot X)_+^{s-|\alpha|}\,
\, \psi(X)\,dX,
\end{eqnarray*}
and this gives the desired result, since~$\psi$ is an arbitrary test function.
\end{proof}

\section*{Appendix B -- Green function computations}

Now, we present
the proof
of Lemma~\ref{GOALEM}. We recall that such result
gives
some precise asymptotics on the boundary behavior of the Green
function of the fractional Laplacian, which in turn
have been exploited
in Appendix~A to obtain precise boundary information
on the solutions of fractional Laplace equations.

\begin{proof}[Proof of Lemma~\ref{GOALEM}]
We remark that the condition $e+\eps\omega\in B_1$ for all $\eps\in(0,\eps_o]$
says that
$$ 1> |e+\eps\omega|^2 = 1 +\eps^2 +2\eps e\cdot\omega$$
and so in particular
\begin{equation}\label{iknAUUU1234AO}
-e\cdot\omega > \frac{\eps}{2}>0.
\end{equation}
{F}rom~\eqref{r0 def}, we have that
\begin{equation}\label{upohlkSR111TGHHHHSS}
r_0(e+\eps\omega,z)=
\frac{\eps\,(-\eps -2\,e\cdot\omega)\,(1-|z|^2)}{|z-e-\eps\omega|^2}.\end{equation}
In particular,
\begin{equation}\label{r0epshe}
r_0(e+\eps\omega,z)\le \frac{3\eps}{|z-e+\eps\omega|^2}.
\end{equation}
Moreover, using a Taylor binomial series,
$$ (t+1)^{-\frac{n}{2}} =\sum_{k=0}^{+\infty} \left( {-n/2}\atop{k}\right)\,t^k$$
and therefore
\begin{equation}\label{7tughjAserJS}
\frac{t^{s-1}}{
(t+1)^{\frac{n}{2}}}=\sum_{k=0}^{+\infty} \left( {-n/2}\atop{k}\right)\,t^{k+s-1}.\end{equation}
Since, by the bounds on the binomial coefficients, we have that
\begin{equation}\label{7tughjAserJS67icoe} \left| \left( {-n/2}\atop{k}\right)\right|\le C\,k^{\frac{n}{2}},\end{equation}
it follows from the root test
that the series in \eqref{7tughjAserJS} is uniformly convergent for any $t$
in a compact subset of $(-1,1)$.
In particular, if we set
\begin{equation}\label{7tughjAserJSr1}
r_1(x,z):= \min\left\{ r_0(x,z),\,\frac12\right\},\end{equation}
we can exchange the integration and summation signs and find that
$$ \int_0^{r_1(x,z)} \frac{t^{s-1}\,dt}{
(t+1)^{\frac{n}{2}}} = \sum_{k=0}^{+\infty} c_k\,\big(
r_1(x,z)\big)^{k+s},$$
with
$$ c_k:= \frac{1}{k+s}\left( {-n/2}\atop{k}\right).$$
Therefore, we have
\begin{equation}\label{ggGG} G(x,z)={\mathcal{G}}(x,z)+g(x,z),\end{equation}
with
\begin{eqnarray*}&&
{\mathcal{G}}(x,z):=
|z-x|^{2s-n}\,\sum_{k=0}^{+\infty} 
c_k\,\big(
r_1(x,z)\big)^{k+s}
\\ {\mbox{and }}&&
g(x,z):=
|z-x|^{2s-n}\int_{r_1(x,z)}^{r_0(x,z)} \frac{t^{s-1}\,dt}{
(t+1)^{\frac{n}{2}}}.\end{eqnarray*}
Notice that $g(x,z)=0$ if $r_0(x,z)\le1/2$. Also, if $r_0(x,z)> 1/2$,
$$ 0\le g(x,z)\le |z-x|^{2s-n}\int_{1/2}^{r_0(x,z)} \frac{t^{s-1}\,dt}{
t^{\frac{n}{2}}} \le 
\left\{
\begin{matrix}
C\, |z-x|^{2s-n} & {\mbox{ if }}n>2s,\\
C\,\log r_0(x,z) & {\mbox{ if }}n=2s,
\\ C \,|z-x|^{2s-n}
\,\big( r_0(x,z)\big)^{s-\frac{n}2}
& {\mbox{ if }}n<2s,
\end{matrix}\right.$$
for some $C>0$. Now we compute this expression in $x:=e+\eps\omega$.
Notice that the 
condition $r_0(e+\eps\omega,z)=r_0(x,z)> 1/2$, combined with \eqref{r0epshe},
says that
\begin{equation}\label{9uihKAsafdgfhSJJ}
|z-e-\eps\omega|^2 \le 9\eps.\end{equation}
As a consequence
\begin{equation}\label{9ihknAS}
\begin{split}
& \left| \int_{B_1} f(z)\,g(e+\eps\omega,z)\,dz\right| \le
\int_{ B_{3\sqrt\eps}(e+\eps\omega)} |f(z)|\,|g(e+\eps\omega,z)|\,dz\\
&\qquad\qquad
\le
\left\{
\begin{matrix}
C\, \int_{ B_{3\sqrt\eps}(e+\eps\omega)} |f(z)|\,|z-e+\eps\omega|^{2s-n}
\,dz & {\mbox{ if }}n>2s,\\
C\,\int_{ B_{3\sqrt\eps}(e+\eps\omega)} |f(z)|\,\log r_0(e+\eps\omega,z) \,dz& {\mbox{ if }}n=2s,
\\ C \,\int_{ B_{3\sqrt\eps}(e+\eps\omega)} |f(z)|\,|z-e+\eps\omega|^{2s-n}
\,\big( 
r_0(e+\eps\omega,z)\big)^{s-\frac{n}2}
\,dz& {\mbox{ if }}n<2s.
\end{matrix}\right.
\end{split}\end{equation}
Now, if $z\in B_{3\sqrt\eps}(e+\eps\omega)$, then $|z-e|\le4\sqrt\eps$ and so
\begin{equation}\label{yhiAJ:0}
|f(z)|\le
C\eps^{\frac\alpha2},\end{equation}
with $C>0$ depending on $f$.
Hence recalling \eqref{r0epshe}, after renaming $C>0$, we deduce from \eqref{9ihknAS} that
\begin{eqnarray*}
\left| \int_{B_1} f(z)\,g(e+\eps\omega,z)\,dz\right| &\le&
\left\{
\begin{matrix}
C\eps^{\frac\alpha2}\, \int_{ B_{3\sqrt\eps}(e+\eps\omega)} |z-e+\eps\omega|^{2s-n}
\,dz & {\mbox{ if }}n>2s,\\
C\eps^{\frac\alpha2}\,\int_{ B_{3\sqrt\eps}(e+\eps\omega)} \log 
\frac{3\eps}{|z-e+\eps\omega|^2}
\,dz& {\mbox{ if }}n=2s,
\\ C \eps^{\frac\alpha2+s-\frac{n}2}\,\int_{ B_{3\sqrt\eps}(e+\eps\omega)} 1\,dz& {\mbox{ if }}n<2s,
\end{matrix}\right.\\ &\le& C\eps^{\frac\alpha2+s}.
\end{eqnarray*}
This and \eqref{ggGG} give that
\begin{equation} \label{VIS-loj-1}
\int_{B_1} f(z)\,G(e+\eps\omega,z)\,dz=
\int_{B_1} f(z)\,{\mathcal{G}}(e+\eps\omega,z)\,dz+o(\eps^s).
\end{equation}
Now we consider the series defining ${\mathcal{G}}$ and we split the contribution
coming from the index $k=0$ from the ones coming from the indices $k\ge1$, namely we write
\begin{equation}\label{VIS-loj-2}
\begin{split}
& {\mathcal{G}}(x,z)=
{\mathcal{G}}_0(x,z)+{\mathcal{G}}_1(x,z)\\
{\mbox{with }}\quad&{\mathcal{G}}_0(x,z):=
\frac{|z-x|^{2s-n}}{s}\,\big(
r_1(x,z)\big)^{s} \\ {\mbox{and }}\quad&
{\mathcal{G}}_1(x,z):=
|z-x|^{2s-n}\,\sum_{k=1}^{+\infty} c_k
\,\big(
r_1(x,z)\big)^{k+s}.
\end{split}\end{equation}
So, we use \eqref{7tughjAserJSr1} and~\eqref{yhiAJ:0} and obtain that
\begin{equation}\label{8yhikARYFHJVGVSIhiIHG}
\begin{split}
& \left| \int_{B_1\cap B_{3\sqrt\eps}(e+\eps\omega) } f(z)\,{\mathcal{G}}_1(e+\eps\omega,z)\,dz \right|
\le 
\int_{ B_{3\sqrt\eps}(e+\eps\omega) } |f(z)|\,{\mathcal{G}}_1(e+\eps\omega,z)\,dz\\
&\qquad\le C\eps^{\frac\alpha2}
\int_{ B_{3\sqrt\eps}(e+\eps\omega) } 
|z-e-\eps\omega|^{2s-n}\,\sum_{k=1}^{+\infty} 
|c_k|\,\big(
r_1(e+\eps\omega,z)\big)^{k+s}
\,dz \\
&\qquad\le C\eps^{\frac\alpha2}
\int_{ B_{3\sqrt\eps}(e+\eps\omega) } 
|z-e-\eps\omega|^{2s-n}\,\sum_{k=1}^{+\infty} |c_k|
\,\big(
1/2\big)^{k+s}
\,dz \\ &\qquad\le C\eps^{\frac\alpha2}
\int_{ B_{3\sqrt\eps}(e+\eps\omega) }
|z-e-\eps\omega|^{2s-n}
\,dz\\ &\qquad\le C\eps^{\frac\alpha2+s},
\end{split}\end{equation}
up to renaming $C>0$.
On the other hand, 
$$ |z|=
|e+\eps\omega+z-e-\eps\omega|
\ge |e+\eps\omega| -|z-e-\eps\omega|\ge 1-\eps-|z-e-\eps\omega|$$
and therefore
$$ |f(z)|\le C\,\big( 1-|z|\big)^\alpha\le
C\,\big( \eps+|z-e-\eps\omega|\big)^\alpha.$$
In particular, if $|z-e-\eps\omega|>3\sqrt\eps$, then
\begin{equation}\label{9uoihknSHBS78A} 
|f(z)|\le 
C\,|z-e-\eps\omega|^\alpha.\end{equation}
Also, using \eqref{r0epshe} and~\eqref{7tughjAserJSr1}, for any $k\ge1$
\begin{equation*}
\begin{split}& \big(
r_1(e+\eps\omega,z)\big)^{k+s}=
\big(r_1(e+\eps\omega,z)\big)^{s+\frac\alpha4}\,\big(r_1(e+\eps\omega,z)\big)^{k-\frac\alpha4}
\\ &\qquad\le \big(
r_0(e+\eps\omega,z)\big)^{s+\frac\alpha4}
\left(
\frac12\right)^{k-\frac\alpha4} \le
\frac{C\,\eps^{s+\frac\alpha4}}{2^k\,
|z-e-\eps\omega|^{2s+\frac\alpha2} }.\end{split}\end{equation*}
This and \eqref{9uoihknSHBS78A} give that, if $z\in
B_1\setminus B_{3\sqrt\eps}(e+\eps\omega)$, then
\begin{eqnarray*}
&& \left| f(z)\,{\mathcal{G}}_1(e+\eps\omega,z)\right|\\ &\le&
C\,|z-e-\eps\omega|^{\alpha+2s-n}\,\sum_{k=1}^{+\infty}
|c_k|\,\big(
r_1(e+\eps\omega,z)\big)^{k+s} \\
&\le& C\eps^{s+\frac\alpha4}\,|z-e-\eps\omega|^{\frac\alpha2-n}\,\sum_{k=1}^{+\infty}
\frac{|c_k|}{2^k},\end{eqnarray*}
and the latter series is convergent, thanks to \eqref{7tughjAserJS67icoe}.
This implies that
\begin{eqnarray*}
&& \left| \int_{B_1\setminus B_{3\sqrt\eps}(e+\eps\omega) } f(z)\,{\mathcal{G}}_1(e+\eps\omega,z)\,dz \right|
\le C\eps^{s+\frac\alpha4}\,\int_{B_1\setminus B_{3\sqrt\eps}(e+\eps\omega) }  
|z-e-\eps\omega|^{\frac\alpha2-n}\,dz\\
&&\qquad\le C\eps^{s+\frac\alpha4}\,\int_{B_1}
|z-e-\eps\omega|^{\frac\alpha2-n}\,dz\le C\eps^{s+\frac\alpha4}.
\end{eqnarray*}
By this and \eqref{8yhikARYFHJVGVSIhiIHG}, we conclude that
$$ \int_{B_1} f(z)\,{\mathcal{G}}_1(e+\eps\omega,z)\,dz=
o(\eps^s).$$
Hence, we insert this information into \eqref{VIS-loj-1} and, recalling \eqref{VIS-loj-2},
we obtain
\begin{equation} \label{VIS-loj-3}
\int_{B_1} f(z)\,G(e+\eps\omega,z)\,dz=
\int_{B_1} f(z)\,{\mathcal{G}}_0(e+\eps\omega,z)\,dz+o(\eps^s).
\end{equation}
Now we define
\begin{eqnarray*}
&& {\mathcal{D}}_1:= \{ z\in B_1 {\mbox{ s.t. }} r_0(e+\eps\omega,z)>1/2\}\\ {\mbox{ and }}
&& {\mathcal{D}}_2:= \{ z\in B_1 {\mbox{ s.t. }} r_0(e+\eps\omega,z)\le 1/2\}.\end{eqnarray*}
If $z\in {\mathcal{D}}_1$, then \eqref{9uihKAsafdgfhSJJ} holds true, and so we can use
\eqref{yhiAJ:0}, to find that
$$ \big| f(z)\,{\mathcal{G}}_0(e+\eps\omega,z)\big|\le
C\eps^{\frac\alpha2}|z-e+\eps\omega|^{2s-n}.$$
Consequently, recalling~\eqref{9uihKAsafdgfhSJJ},
\begin{eqnarray*}
&&\left|
\int_{{\mathcal{D}}_1} f(z)\,{\mathcal{G}}_0(e+\eps\omega,z)\,dz\right|\le
C\eps^{\frac\alpha2}\,
\int_{B_{3\sqrt\eps}(e+\eps\omega)} |z-e+\eps\omega|^{2s-n}\,dz=C\eps^{\frac\alpha2+s},
\end{eqnarray*}
up to renaming $C>0$ once again. In this way,
formula \eqref{VIS-loj-3} reduces to
\begin{equation} \label{VIS-loj-4}
\int_{B_1} f(z)\,G(e+\eps\omega,z)\,dz=
\int_{{\mathcal{D}}_2} f(z)\,{\mathcal{G}}_0(e+\eps\omega,z)\,dz+o(\eps^s).
\end{equation}
Now, by \eqref{7tughjAserJSr1} and \eqref{upohlkSR111TGHHHHSS},
if $z\in{{\mathcal{D}}_2}$,
$$ {\mathcal{G}}_0(e+\eps\omega,z) =
\frac{|z-e-\eps\omega|^{2s-n}}{s}\,\big(
r_0(e+\eps\omega,z)\big)^{s} =
\frac{ \eps^s\,(-\eps -2\,e\cdot\omega)^s\,(1-|z|^2)^s }{
s\,|z-e-\eps\omega|^n }
.$$
Hence, \eqref{VIS-loj-4} gives that
\begin{equation} \label{VIS-loj-5}
\begin{split}&\lim_{\eps\searrow0}\eps^{-s}\int_{B_1} f(z)
\,G(e+\eps\omega,z)\,dz\\
=\,&\lim_{\eps\searrow0}
\int_{ \{
{2\eps\,(-\eps -2\,e\cdot\omega)\,(1-|z|^2)} \le{|z-e-\eps\omega|^2}
\} } f(z)\,
\frac{(-\eps -2\,e\cdot\omega)^s\,(1-|z|^2)^s }{
s\,|z-e-\eps\omega|^n }\,dz.\end{split}
\end{equation}
Now, we show the following uniform integrability condition: we set
\begin{equation*}
F_\eps(z):= \left\{
\begin{matrix}
f(z)\,
\displaystyle\frac{(-\eps -2\,e\cdot\omega)^s\,(1-|z|^2)^s }{
s\,|z-e-\eps\omega|^n } & {\mbox{ if }}
{2\eps\,(-\eps -2\,e\cdot\omega)\,(1-|z|^2)} \le{|z-e-\eps\omega|^2},
\\ \,&\,
\\ 0 & {\mbox{ otherwise}},
\end{matrix}
\right.\end{equation*}
and we prove that for any $\eta>0$ there exists $\delta>0$ 
(depending on $\eta$, $e$ and $\omega$, but
independent of $\eps$) such that,
for any $E\subset\R^{\dpic}$ with $|E|\le\delta$, we have
\begin{equation}\label{9iAJJJJ1425637849t}
\int_{B_1\cap E} \big| F_\eps(z)\big|\,dz\le\eta.
\end{equation}
To this aim, we take $E$ as above and $$\rho:= c_\star\,\eps,$$ with $c_\star\in\left(0,
\frac{1}{10}\right)$ to be conveniently chosen in the sequel (also in dependence of $\omega$
and $e$),
and we set $E_1:=E\cap B_{\rho}(e+\eps\omega)$,
$E_2:=E\setminus E_1$.

We claim that
\begin{equation}\label{iHKA12345ykscb:1}
{\mbox{$E_1$
is empty.}}\end{equation}
For this, we argue by contradiction: if there existed $z\in E_1$, then
\begin{eqnarray*}
&& {\eps\,(-\,e\cdot\omega)\,(1-|z|^2)} 
\le
{2\eps\,(-\eps -2\,e\cdot\omega)\,(1-|z|^2)} 
\le{|z-e-\eps\omega|^2}\le\rho^2,
\end{eqnarray*}
if $\eps$ is small enough in dependence of the fixed $e$ and $\omega$ (recall \eqref{iknAUUU1234AO}),
and thus \begin{equation}\label{iHKA12345ykscb:2}
1-|z|^2 \le \frac{C\rho^2}{\eps},\end{equation}
with $C>0$ also depending on $e$ and $\omega$. 
On the other hand, we have that~$E_1\subseteq B_{\rho}(e+\eps\omega)$,
therefore
$$ |z|\le |e+\eps\omega| +|z-e-\eps\omega|\le \sqrt{1+\eps^2+2\eps e\cdot\omega}+\rho
\le 1 -\frac{-\eps e\cdot\omega}{10}+C\eps^2+\rho,$$
and so
$$ |z|^2 \le 1 -\frac{\eps e\cdot\omega}{5}+C\eps^2+\rho^2.$$
This is a contradiction with \eqref{iHKA12345ykscb:2} if $c_\star$ is appropriately small
and so \eqref{iHKA12345ykscb:1} is proved.

So, from now on,~$c_\star$ is fixed suitably small. 
We observe that if $z\in E_2$ then
$$ |z-e-\eps\omega| \ge \rho = c_\star \eps,$$
and consequently
\begin{equation}\label{9iAJJJJ1425637849t=2}
\int_{B_1\cap E_2} \big| F_\eps(z)\big|\,dz\le 
\int_{ {B_1\cap E}\atop{ \{ |z-e-\eps\omega| \ge c_\star \eps\}}}
\frac{C\,(1-|z|)^{s+\alpha} }{
|z-e-\eps\omega|^n }\,dz.
\end{equation}
Now, we distinguish two cases, either $\delta\le \eps^{2n}$
or $\delta>\eps^{2n}$.
If $\delta\le \eps^{2n}$, we
use \eqref{9iAJJJJ1425637849t=2} to get that
\begin{equation}\label{9iAJJJJ1425637849t=3}
\int_{B_1\cap E_2} \big| F_\eps(z)\big|\,dz\le
\int_{ {B_1\cap E}\atop{ \{ |z-e-\eps\omega| \ge c_\star \eps\}}}
\frac{C }{
\eps^n }\,dz\le \frac{C \delta}{
\eps^n } \le C\,\sqrt\delta.\end{equation}
If instead 
\begin{equation}\label{u9ojkSsdfHH}
\delta>\eps^{2n},\end{equation} we observe that
$$ |z-e-\eps\omega|\ge 1-|z|-\eps$$
and so we deduce from \eqref{9iAJJJJ1425637849t=2} that
\begin{equation}\label{8YIHAHHH}
\begin{split}
& \int_{B_1\cap E_2} \big| F_\eps(z)\big|\,dz\le
\int_{ {B_1\cap E}\atop{ \{ |z-e-\eps\omega| \ge c_\star\eps\}}}
\frac{C\,(|z-e-\eps\omega|+\eps)^{s+\alpha} }{
|z-e-\eps\omega|^n }\,dz\\
&\qquad \le
C\,\int_{ {B_1\cap E}\atop{ \{ |z-e-\eps\omega| \ge c_\star \eps\}}}
\frac{|z-e-\eps\omega|^{s+\alpha} }{
|z-e-\eps\omega|^n }\,dz+
C\,\int_{ {B_1\cap E}\atop{ \{ |z-e-\eps\omega| \ge c_\star \eps\}}}
\frac{\eps^{s+\alpha} }{
|z-e-\eps\omega|^n }\,dz\\&\qquad=:I_1+I_2.
\end{split}\end{equation}
To estimate $I_1$, we split into
\begin{eqnarray*}&&
I_{1,1}:= C\,\int_{ {B_1\cap E}\atop{ \{ c_\star \eps \le |z-e-\eps\omega| \le\delta^{1/{2n}} \}}}
|z-e-\eps\omega|^{s+\alpha-n} 
\,dz\\ {\mbox{and }}&&
I_{1,2}:=
C\,\int_{ {B_1\cap E}\atop{ \{ |z-e-\eps\omega| >\delta^{1/{2n}}\}}}
|z-e-\eps\omega|^{s+\alpha-n} \,dz
.\end{eqnarray*}
Using polar coordinates, we find that
\begin{equation}\label{UOJAJJJA5678JA12345}
I_{1,1}\le C\, \int_{c_\star \eps}^{ \delta^{1/{2n}} } 
t^{n-1}\,t^{s+\alpha-n}
\,dt \le C\, \left[\left( \delta^{\frac1{2n}}\right)^{s+\alpha}-
\left( c_\star \eps\right)^{s+\alpha}\right]\le C\,\delta^{\frac{s+\alpha}{2n}}.
\end{equation}
In addition,
$$ I_{1,2}\le
C\,\int_{ {B_1\cap E}\atop{ \{ |z-e-\eps\omega| >\delta^{1/{2n}}\}}}
|z-e-\eps\omega|^{s-n}\,dz \le 
C\,\int_E
\delta^{{\frac{s-n}{2n}} }\,dz \le C\delta^{1+{\frac{s-n}{2n}}}=
C\delta^{{\frac{s+n}{2n}}}.$$
This and \eqref{UOJAJJJA5678JA12345} say that
\begin{equation}\label{UOJAJJJA5678JA12345BIS}
I_{1}\leq 
C\,\delta^{\frac{s+\alpha}{2n}}+C\delta^{{\frac{s+n}{2n}}}.
\end{equation}
Moreover,
$$ I_2 \le C\eps^{s+\alpha}\,\int_{c_\star\eps}^2\frac{t^{n-1}}{t^N}\,dt\le
C\eps^{s+\alpha}\,|\log\eps|\le C\eps^s\le C\delta^{\frac{s}{2n}},$$
thanks to \eqref{u9ojkSsdfHH}. Hence, using this and \eqref{UOJAJJJA5678JA12345BIS},
and recalling \eqref{8YIHAHHH}, we obtain that
$$ \int_{B_1\cap E_2} \big| F_\eps(z)\big|\,dz\le
I_1+I_2\le C\,\delta^{\frac{s+\alpha}{2n}}+C\delta^{{\frac{s+n}{2n}}}+
C\delta^{\frac{s}{2n}}.$$
We now combining this estimate, which is coming from the case in \eqref{u9ojkSsdfHH},
with \eqref{9iAJJJJ1425637849t=3}, which was coming from the complementary case,
and we see that, in any case,
$$ \int_{B_1\cap E_2} \big| F_\eps(z)\big|\,dz\le C\,\delta^\kappa,$$
for some $\kappa>0$. {F}rom this and
\eqref{iHKA12345ykscb:1}, we obtain that
$$ \int_{B_1\cap E} \big| F_\eps(z)\big|\,dz\le C\,\delta^\kappa,$$
Then, choosing $\delta$ suitably small with respect to $\eta$,
we establish \eqref{9iAJJJJ1425637849t}, as desired.

Notice also that $F_\eps$ converges pointwise 
to $ f(z)\, \frac{(-2\,e\cdot\omega)^s\,(1-|z|^2)^s 
}{ s\,|z-e|^n }$. Hence, using \eqref{VIS-loj-5},
\eqref{9iAJJJJ1425637849t} and the Vitali Convergence Theorem, we conclude that
\begin{eqnarray*}
\lim_{\eps\searrow0}\eps^{-s}\int_{B_1} f(z)\,G(e+\eps\omega,z)\,dz &=&
\lim_{\eps\searrow0} \int_{B_1} F_\eps(z)\,dz\\
&=& \int_{B_1} f(z)\, \frac{(-2\,e\cdot\omega)^s\,(1-|z|^2)^s
}{ s\,|z-e|^n }\,dz,
\end{eqnarray*}
which establishes \eqref{GOA}.
\end{proof}

\end{appendix}

\end{document}